\documentclass{amsart}
\usepackage{amsfonts,amssymb,amsmath,amsthm}
\usepackage{url}
\usepackage{enumerate}

\newtheorem{theorem}{Theorem}
\newtheorem{lemma}[theorem]{Lemma}

\newtheorem{corollary}[theorem]{Corollary}
\theoremstyle{definition}
\newtheorem{definition}[theorem]{Definition}
\newtheorem{remark}[theorem]{Remark}

\newcommand{\DKL}[2]{\operatorname{KL}(#1\parallel#2)}

\numberwithin{equation}{section}

\newcommand{\kmix}{\ensuremath{k\textnormal{\normalfont-mix}}}
\newcommand{\transpose}{^{\mathsf{T}}}                                   %\DeclareMathOperator{\tr}{Tr}

\usepackage{mathtools,url, bbm}
\usepackage{amsmath,amssymb,amstext}
\usepackage{listings}

\newcommand{\DTV}[2]{\|#1 - #2\|_{TV}}

\newcommand{\eps}{\varepsilon}

\DeclareMathOperator{\tr}{Tr}

\DeclareMathOperator{\vc}{VC-dim}

\DeclareMathOperator{\nn}{NN}
\newcommand{\fcal}{\mathcal F}
\newcommand{\gcal}{\mathcal G}
\newcommand{\acal}{\mathcal A}
\newcommand{\ycal}{\mathcal Y}
\newcommand{\pcal}{\mathcal P}
\newcommand{\ncal}{\mathcal N}

\newcommand{\R}{\mathbb{R}}
\DeclareMathOperator{\kl}{KL}

\usepackage[dvipsnames]{xcolor}

\author{Hassan Ashtiani \and Abbas Mehrabian}
\address{
	Hassan Ashtiani : University of Waterloo\\
	Abbas Mehrabian: McGill University}
\email{abbasmehrabian@gmail.com, mhzokaei@uwaterloo.ca}
\thanks{Abbas Mehrabian was supported by a CRM-ISM postdoctoral fellowship.}

\begin{document}
%Version of \today

\title[Some techniques in density estimation]{Some techniques in density estimation}

\begin{abstract}
Density estimation is an interdisciplinary topic at the intersection of statistics, theoretical computer science and machine learning.
We review some old and new techniques \cite{ashtiani2017agnostic,ashtiani2017sample,onedimensional,spherical} for bounding the sample complexity of estimating densities of continuous distributions, focusing on the class of mixtures of Gaussians and its subclasses.
In particular, we review the main techniques used to prove the new sample complexity bounds for mixtures of Gaussians by Ashtiani, Ben-David, Harvey, Liaw, Mehrabian, and Plan.
\end{abstract}
\maketitle

%\hassan{For a survey, we will probably need to be more comprehensive? I think at the moment we have a survey of mostly our own results, which is very nice. We may need to find a good way to sell it though}

\tableofcontents

\section{Introduction}
These days unsupervised learning is very popular due to the amount of available unlabeled data.
%A core task in unsupervised learning is to understand the distribution of the data, which is known to significantly improve the efficiency of learning algorithms.
The general goal in unsupervised learning is to find structure in the data.
This `structure' can be the clusters in the data,
the principal components of the data, or 
the intrinsic dimension of the data, and so on.
Distribution learning (also known as density estimation) is the task of explicitly estimating the 
distribution underlying the data, which can then be explored to find the desired structure, or to generate new data.
We mention two examples.

The first example, taken from~\cite{jordan96neural}, is anomaly detection: interpreting X-ray images (mammograms)  for detecting breast cancer.
In this case the training data consists of normal (non-cancerous) images; a probability density function $\mu:\mathbb{R}^d\to\mathbb{R}$ is learned from the data.
When a new input $x'$ is presented to the system, a high value for $\mu(x')$ indicates a normal image, while a low value indicates a novel input, which might be characteristic of an abnormality; the patient is then referred to a clinician for further examination.
The second example, taken from~\cite{deep_learning_book}, is synthesis and sampling, or generative models: in many cases we would like to generate new examples that are similar to those in the training data, e.g., in media applications, where it can be expensive or boring for an artist to generate large volumes of content by hand. 
Given the training data, the algorithm  
estimates a probability density function $\mu:\mathbb{R}^d\to\mathbb{R}$ that models the data, and then generates new examples according to this distribution.
For example, video games can automatically
generate (random but reasonable) textures for large objects or landscapes, rather than requiring an artist to manually colour each pixel.

For supervised learning (in particular, classification problems), there are by now a variety of mathematical tools to understand the hardness of the problem (VC-dimension, Rademacher complexity, covering numbers, margins, etc., see~\cite{AB99,Shalev-Shwartz2014}).
We lack such a satisfactory mathematical understanding in the case of unsupervised learning (in particular, distribution learning); determining the sample complexity of learning with respect to a general class of distributions is an open problem (see~\cite[Open Problem~15.1]{Diakonikolas2016}).
%My overarching research goal in this area is to develop a mathematical framework for distribution learning.

More specifically, distribution learning refers to the following task:
given data generated from an unknown target probability distribution $\mu$, find a distribution $\widehat{\mu}$ that is `close' to $\mu$.
To define this problem more precisely, one needs to specify:

\begin{enumerate}
\item What is assumed about the target distribution? 
This question is more pertinent than ever in this era of large high-dimensional data sets.
Typically one assumes the target belongs to some class of distributions, or it is close to some distribution in that class.
\item What does `close' mean? There are various statistical measures of closeness, e.g., the Kullback-Leibler divergence, $L_1$ and $L_2$ distances (see~\cite[Chapter~5]{devroye_book} for a discussion).
\item How is data sampled from the distribution? One  usually assumes access to i.i.d.\ data, but in some settings other models such as Markov Chain-based sampling may be more appropriate.
\end{enumerate}

%\hassan{mention uniform/minimax vs ...?}
Once the above questions are answered, we have a well defined problem, for which we can propose algorithms.
Such an algorithm is evaluated using two metrics:
(i)~the sample complexity, i.e., the number of samples  needed to guarantee a small error, and 
(ii)~the computational complexity, or the running time of the algorithm.

In this survey, we assume the target class consists of mixtures of Gaussians in high dimensions, or is a subclass of this class.
We focus on the $L_1$ distance as the measure of closeness, and we assume i.i.d.\ sampling.
Our goal is to give bounds for the \emph{sample complexity} for distribution learning, or density estimation.
We shall use these two phrases interchangeably here; distribution learning (or PAC-learning of distributions) is usually used in the computer science/machine learning community and is a broader term, whereas density estimation is usually used in the statistics literature (see \cite[Section~2]{Diakonikolas2016} for a discussion).

The reason for this choice is that recently this problem has attracted much attention again, and many results have been proved during the last few years (see~\cite{ashtiani2017agnostic,ashtiani2017sample,onedimensional,gaussian_mixture,spherical}).

The literature on density estimation is vast and we have not tried to be comprehensive.
We shall just review some techniques that have been particularly successful in proving rigorous bounds for sample complexity of learning mixtures of Gaussians.
The reader is referred to~\cite{Diakonikolas2016} for a broader, recent survey.
For a general, well written introduction to density estimation, read~\cite{devroye_book}.
This survey is based on the papers
\cite{ashtiani2017agnostic,ashtiani2017sample,onedimensional,spherical};
the reader is referred to the original papers for full proofs.
Most of the material in Section~\ref{sec:vc} also appears in~\cite{devroye_book}.

In Section~\ref{sec:formal} we set up our notation. In each subsequent section we present one technique and demonstrate it by showing a bound on the sample complexity of learning a particular class of distributions.
Concluding remarks appear in Section~\ref{sec:remarks}.

\section{The formal framework}
\label{sec:formal}
A \emph{distribution learning method} or \emph{density estimation method} is an algorithm that takes as input an i.i.d.\ sample generated from a distribution $g$, and outputs (a description) of a distribution $\hat{g}$ as an estimation for $g$.
Furthermore, we assume that $g$ belongs to some known class $\mathcal{F}$ of distributions, but $\widehat{g}$ is not required to belong to $\fcal$
(if it does, then the method is called a `proper' learner). 

We only consider continuous distributions in this survey,
and so we identify a `probability distribution' by its `probability density function.' 
Let $Z$ be a Euclidean space,
and let $f_1$ and $f_2$ be two distributions defined over the Borel $\sigma$-algebra $\mathcal{B}\subseteq 2^Z$. 
The total variation distance between $f_1$ and $f_2$ is defined by
\[\|f_1- f_2\|_{TV} = 
\sup_{B\in \mathcal{B}}|f_1(B) - f_2(B)| = 
\sup_{B\in \mathcal{B}}\left|\int_B f_1 - \int_B f_2\right| =
\frac{1}{2}\|f_1 - f_2\|_1 \:,
\]
where 
\(\|f\|_1\coloneqq \int_Z |f(x)|\mathrm{d} x\) 
is the $L_1$ norm of $f$.
In the following definitions,  $\mathcal{F}$ is a class of probability distributions, and $g$ is an arbitrary distribution.
The total variation distance and the $L_1$ distance are within constant factor of each other, and we generally do not worry about constants in this survey, so we will use them interchangeably, except when a confusion might occur.

We write $X \sim g$ to denote the random variable (or random vector) $X$ has distribution $g$, and we write $S\sim g^m$ to mean that $S$ is an i.i.d.\ sample of size $m$ generated from $g$.

\begin{definition}[$\eps$-approximation, $\eps$-close]
	A distribution $\hat{g}$ is 
	an \emph{$\eps$-approximation} for $g$,
	or is \emph{$\eps$-close} to $g$,
	 if $ \|\hat{g}- g\|_1 \leq \eps$.
\end{definition}

\begin{definition}[density estimation method, sample complexity]\label{def:density_estimation}
	A \emph{density estimation method} for $\mathcal{F}$ has \emph{sample complexity} $m_{\mathcal{F}}(\eps, \delta)$ if, for any distribution $g\in\mathcal F$ and any $\eps, \delta \in(0,1)$, given $\eps$, $\delta$, and an i.i.d.\ sample of size $m_{\mathcal{F}}(\eps, \delta)$ from $g$, with probability at least $1-\delta$ outputs an $\eps$-approximation of $g$.
\end{definition}

In the machine learning literature, such a density estimation method for class $\fcal$ is called a `PAC distribution learning method for $\fcal$ in the realizable setting,' or an `$\fcal$-leaner,' with sample complexity $m_{\mathcal{F}}(\eps, \delta)$.
We also say we can `learn class $\fcal$ with 
$m_{\mathcal{F}}(\eps, \delta)$ samples'.
Typically the dependence on $\delta$ is logarithmic and hence non-crucial, and sometimes we may just say the sample complexity is $m_{\mathcal{F}}(\eps)$, ignoring its dependence on $\delta$ (this means we take, e.g., $\delta=1/3$). Note that the sample complexity should not depend on the specific underlying probability distribution $g$, but should uniformly hold for \emph{all} $g\in \fcal$. This uniform notion of learning is sometimes called \emph{minimax} density estimation in the statistics literature.

%\begin{definition}[PAC-Learning Distributions, Agnostic Setting]
%	For $C>0$, a distribution learning method is called a \emph{$C$-agnostic PAC-learner for $\mathcal{F}$} with sample complexity $m_{\mathcal{F}}^C(\eps, \delta)$, if for all distributions $g$ and all $\eps, \delta \in(0,1)$, given $\eps$, $\delta$, and a sample of size $m_{\mathcal{F}}^C(\eps, \delta)$, with probability at least $1-\delta$ outputs an $(\eps, C)$-approximation of $g$.
%\end{definition}
%We sometimes say a class can be ``$C$-learned in the agnostic setting'' to indicate the existence of a $C$-agnostic PAC-learner for the class. Note that the case $C>1$ is sometimes called \emph{semi}-agnostic learning.

Let 
\[
\Delta_k \coloneqq \left\{ (w_1,\dots,w_k) : w_i\geq 0, \sum w_i=1\right\}
\]
denote the $k$-dimensional simplex.

\begin{definition}[$k$-mix($\mathcal{F}$)]
The class of \emph{$k$-mixtures} of $\mathcal{F}$, written \emph{$k$-mix($\mathcal{F}$)}, is defined as 
	$$k\textnormal{-mix}(\fcal) \coloneqq \left\{\sum_{i=1}^{k}w_{i}f_{i}: (w_1,\dots,w_k)\in \Delta_k ,
	f_1,\dots,f_k\in\mathcal F
	\right\}.$$
\end{definition}

%\hassan{notation for Gaussian density vs Gaussian random variable}
A one-dimensional Gaussian random variable with mean $\mu$ and variance $\sigma^2$ is denoted by $N(\mu,\sigma^2)$.
Let $d$ be a positive number, denoting the dimension.
A $d$-dimensional Gaussian with mean $\mu\in \R^d$ and (positive semidefinite) covariance matrix $\Sigma \in \R^{d\times d}$ is a probability distribution over $\R^d$, denoted $\ncal(\mu,\Sigma)$, with probability density function
\[
\ncal(\mu,\Sigma) (x) \coloneqq \exp ( - (x-\mu)^T \Sigma^{-1} (x-\mu)/2)/\sqrt{2\pi \det \Sigma}.
\]
%One way to generate from this distribution is as follows.
%Let $g_1,\dots,g_d \sim N(0,1)$ be i.i.d.
%Then
%\[\mu + \frac {\Sigma^{1/2}} {\sqrt{d}}  \begin{pmatrix} g_1 \\ \vdots \\ g_d \end{pmatrix} \sim N(\mu, \Sigma).\]
A random variable with density $\ncal(\mu,\Sigma)$ is denoted by $N(\mu,\Sigma)$.
Let $\gcal_{d,1}$ denote the class of $d$-dimensional Gaussian distributions,
and let $\gcal_{d,k}$ denote the class of $k$-mixtures of $d$-dimensional Gaussian distributions: $\gcal_{d,k} \coloneqq k\textnormal{-mix}(\gcal_{d,1})$.

If $\Sigma$ is a diagonal matrix, then $\ncal(\mu,\Sigma)$ is called an \emph{axis-aligned} Gaussian, since in this case the eigenspace of $\Sigma$ coincide with the standard basis.
Let $\acal_{d,1}$ denote the class of $d$-dimensional axis-aligned Gaussian distributions,
and let $\acal_{d,k} \coloneqq k\textnormal{-mix}(\acal_{d,1})$.

%\hassan{table: reference Devroye for Gaussian (the first two rows of the table?}

\renewcommand\arraystretch{1.5} 
\begin{table}
\begin{tabular}{c | c | c | c}
Distribution family & Bound on sample complexity & Reference & Section \\\hline
$\gcal_{d,1}$ & ${O(d^2 / \eps^2)}$ & \cite{ashtiani2017sample} &  \ref{sec:vc} \\
$\acal_{d,1}$ & $O(d / \eps^2)$ & \cite{ashtiani2017sample} &  \ref{sec:vc} \\
$\gcal_{1,k}=\acal_{1,k}$ & $\widetilde{O}(k/\eps^2)$& \cite{onedimensional} & \ref{sec:onedimensional}\\
$\gcal_{d,k}$ & $\widetilde{O}(kd^2 / \eps^2)$ & \cite{ashtiani2017agnostic} &  \ref{sec:compression} \\
$\acal_{d,k}$ & $\widetilde{O}(kd /  \eps^2)$ & \cite{ashtiani2017agnostic} &  \ref{sec:compression} \\ 
\hline
$\gcal_{d,1}$&$\widetilde\Omega(d^2 /  \eps^2)$&\cite{ashtiani2017agnostic}&\ref{sec:lower}\\
$\acal_{d,1}$&$\widetilde\Omega(d /  \eps^2)$&\cite{spherical}&\ref{sec:lower}\\
$\gcal_{1,k}=\acal_{1,k}$ & $\widetilde\Omega(k/\eps^2)$& \cite{spherical} & \ref{sec:lower}\\
$\gcal_{d,k}$&$\widetilde\Omega(kd^2 /  \eps^2)$&\cite{ashtiani2017agnostic}&\ref{sec:lower}\\
$\acal_{d,k}$&$\widetilde\Omega(kd /  \eps^2)$&\cite{spherical}&\ref{sec:lower}\\
\hline
\end{tabular}
\vspace{5mm}
\caption{Results covered in this survey. The top 5 rows are upper bounds, and the bottom 5 rows are  lower bounds.}\label{table}
\end{table}

We demonstrate some of the techniques used in density estimation
by proving the upper and lower bounds in Table~\ref{table}.
For proving the upper bounds, we provide a density estimation method.
For the lower bounds, we show that \emph{any} density estimation method for the corresponding class  must use at least the given number of samples.
Throughout the survey, $\widetilde{O}, \widetilde{\Omega}, \widetilde{\Theta}$ allow for polylogarithmic factors, and $[n]\coloneqq\{1,\dots,n\}$.
We say a bound is tight if it is tight up to polylogarithmic factors.

\section{Sample complexity upper bounds via VC-dimension}
\label{sec:vc}

The VC-dimension of a set system,
first introduced by Vapnik and Chervonenkis,
has applications in diverse areas such as
graph theory, discrete geometry~\cite{epsnet}, and the theory of empirical processes~\cite{tal},
and is known to precisely capture the sample complexity of learning in the setting of binary classification~\cite{AB99,Shalev-Shwartz2014}.
In this section we show it can also be used to give upper bounds for sample complexity of density estimation.
The methods of this section have been developed in~\cite{devroye_book}, which is the first place where VC-dimension is used for bounding the sample complexity of density estimation.
The main results of this section appear in~\cite{ashtiani2017sample}.

We start by defining the Vapnik-Chervonenkis dimension (VC-dimension for short) of a set system.

\begin{definition}[VC-dimension]
Let $H$ be a family of subsets of a set $Z$.
We say a set $A\subseteq Z$ is shattered if,
for \emph{any} $B\subseteq A$, there exists some $C\in H$ such that $A \cap C = B$.
The \emph{VC-dimension} of $H$, denoted $\vc(H)$, is the size of the largest shattered set.
\end{definition}

In this section we show an upper bound of $O(d^2/\eps^2)$ for learning $\gcal_{d,1}$,
and an upper bound of $O(d/\eps^2)$ for learning $\acal_{d,1}$.
The plan is to first connect the sample complexity of learning an arbitrary class $\fcal$ to the VC-dimension of a class of a related set system, called the Yatracos class of $\fcal$ (this is done in Theorem \ref{thm:distributionVC}), and then provide upper bounds on VC-dimension of this set system.
Let $Z$ be an arbitrary set, which will be the domain of our probability distributions.

\begin{definition}[$\mathcal{A}$-distance] Let $\mathcal{A}\subseteq 2^Z$, and let $p$ and $q$ be two probability distributions over $Z$. The \emph{$\mathcal{A}$-distance} between $p$ and $q$ is defined as
	$$
	\|p-q \|_{\mathcal{A}} \coloneqq \sup_{A\in \mathcal{A}} |p(A) - q(A)|.
	$$
\end{definition}

%\begin{definition}[Total Variation Distance] Let $p$ and $q$ be two probability measures over $X$. Then total-variation distance between these measures is defined as follows.
%$$
%\|p-q \|_{TV} =  \sup_{A\in 2 ^ X} |p(A) - q(A)| = \frac{1}{2} \|p-q \|_1
%$$
%\end{definition}

\begin{definition}[empirical distribution] Let $S = (x_i)_{i=1}^{m}$ be a sequence of members of $Z$. The \emph{empirical distribution} corresponding to this sequence is defined by $\hat{p}_S(A) \coloneqq  \frac 1 m |A \cap \{x_1,\dots,x_m\}|$ for any $A\subseteq Z$. 
\end{definition}
%Furthermore, let $p$ be a probability measure over $X$. Then, we denote by $\hat{p}_m$ the empirical distribution corresponding to a randomly (i.i.d.\) generated sequence of size $m$ from $p$.

The following lemma is a well known refinement of the uniform convergence theorem of~\cite{vc}, due to Talagrand~\cite{tal};
we use the wording of~\cite[Theorem~4.9]{AB99}.
\begin{lemma}[uniform convergence theorem]\label{lemma:vc}
	Let $p$ be a probability distribution over $Z$. 
	%Let $\hat{p}_m$ be the corresponding empirical distribution function. 
	Let $\mathcal{A} \subseteq 2^Z$ and let $v$ be the VC-dimension of $\mathcal{A}$. Then, there exist universal positive constants $c_1,c_2,c_3$ such that
	$$\mathbf{Pr}_{S\sim p^m} \{\|p-\hat{p}_S \|_{\mathcal{A}} \geq \eps \}\leq 
	\exp(c_1 + c_2 v - c_3 m \eps^2)\:.$$
\end{lemma}

\begin{definition}[Yatracos class] \label{def_yatrocas}
	For a class $\fcal$ of functions from $Z$ to $\mathbb{R}$,
	the associated \emph{Yatracos class} is the family of subsets of $Z$ defined as
	\[
	\mathcal{Y}(\fcal) \coloneqq
	\left\{
	\{x\in Z : f_1(x)\geq f_2(x)\}
	\textnormal{ for some }
	f_1,f_2\in\fcal.
	\right\}
	\]
\end{definition}
% Let $f:X\rightarrow \mathbb{R}$ be a real-valued function. The epigraph of $f$ is defined by $EPI(f) = \{x\in X : f(x)\geq 0\}$. 
% Furthermore, the epigraph of a class of functions $\mathcal{F}$ is defined by $EPI(\mathcal{F}) = \{EPI(f) : f \in \mathcal{F}\}$.

Observe that if $f,g \in \fcal$ then
\(\|f - g \|_{TV} = \|f - g \|_{\mathcal{Y}(F)}\).
To see this, let $A \coloneqq \{x : f(x) \geq g(x)\}\in\ycal({\fcal})$, and observe that
\begin{align*}
\|f - g \|_{TV} & = \frac 1 2
\|f-g\|_{1} = 
\frac 1 2 \int |f(x)-g(x)| dx 
= \int_{A} (f(x) - g(x)) dx  \\
&= \int_{A} f(x) dx -
\int_{A} g(x) dx = |f(A)-g(A)| \leq \|f - g\|_{\ycal(\fcal)} .
\end{align*}
The other direction, namely
$\|f - g \|_{TV} \geq \|f - g\|_{\ycal(\fcal)}$,
follows from the definition of the total variation distance.

\begin{definition}[empirical Yatracos minimizer]\label{eym}
	Let $\mathcal{F}$ be a class of distributions over domain $Z$. 
	The \emph{empirical Yatracos minimizer} is a function
	$L^{\mathcal{F}}:\bigcup_{m=1}^{\infty} Z^m \to \mathcal{F}$ defined as
	$$L^{\mathcal{F}}(S) =\arg\min_{q\in \mathcal{F}} \|q - \hat{p}_S \|_{\mathcal{Y}(\mathcal{F})}.$$
\end{definition}

If the argmin is not unique, we may choose one arbitrarily.

\begin{theorem}[density estimation via empirical Yatracos minimizer]
	\label{thm:distributionVC}
	Let $\mathcal{F}$ be a class of probability distributions, and let $S\sim p^m$, where $p\in \fcal$. Then, with probability at least $1 - \delta$ we have
	$$\|p - L^{\mathcal{F}}(S)\|_{TV} \leq c \sqrt{\frac{v + \log \frac{1}{\delta}}{m}}       $$
	where $v$ is VC-dimension of $\mathcal{Y}(\mathcal{F})$, and $c$ is a universal constant.
\end{theorem}

We remark that variants of this result, without explicit dependence on the failure probability, is proved
implicitly in \cite{devroye_book} and 
also appears explicitly in \cite[Lemma~6]{logconcave}.

\begin{proof}
%	Since
%	$L^{\mathcal{F}}(S), p\in\fcal$ we have
%	$
%	\|L^{\mathcal{F}}(S) - p \|_{TV} =\|L^{\mathcal{F}}(S) - p \|_{\mathcal{Y}(\mathcal{F})}$.
%	By Lemma~\ref{lemma:vc},
%	with probability $\geq1-\delta$
%	we have 
%	$\|p - \hat{p}_S \|_{\mathcal{Y}(\mathcal{F})} \leq
%	c \sqrt{(v+\log \frac{1}{\delta})/{m}}$ for some universal constant $c$.
%	Also, since $L^{\mathcal{F}}(S)$ is the empirical minimizer of the $\mathcal{Y}(\mathcal{F})$-distance, we have
%	$\|L^{\mathcal{F}}(S) - \hat{p}_S \|_{\mathcal{Y}(\mathcal{F})}
%	\leq
%	\|p - \hat{p}_S \|_{\mathcal{Y}(\mathcal{F})}$.
%	The proof follows from these facts combined with an application of the triangle inequality:
We have
	\begin{align*}
	\| p - L^{\mathcal{F}}(S)\|_{TV}  &=\|L^{\mathcal{F}}(S) - p \|_{\mathcal{Y}(\mathcal{F})}  
	\leq \|L^{\mathcal{F}}(S) - \hat{p}_S \|_{\mathcal{Y}(\mathcal{F})} + \|\hat{p}_S - p \|_{\mathcal{Y}(\mathcal{F})}   \\
	&\leq 2 \|p - \hat{p}_S \|_{\mathcal{Y}(\mathcal{F})}  \leq c \sqrt{\frac{v + \log \frac{1}{\delta}}{m}} \:.
	\end{align*}
	The equality is because $L^{\mathcal{F}}(S), p\in\fcal$.
	The first inequality is the triangle inequality.
	The second inequality is because
	$L^{\mathcal{F}}(S)$ is the empirical minimizer of the $\mathcal{Y}(\mathcal{F})$-distance.
	The third inequality holds by Lemma~\ref{lemma:vc}
	with probability $\geq1-\delta$.
\end{proof}

\begin{corollary}\label{cor:vcsample}
For any class $\fcal$, the sample complexity for learning $\fcal$ is bounded by 
$O\left(  (\vc(\ycal(\fcal)) + \log (1/\delta)) /\eps^2 \right)$.
\end{corollary}

%Theorem \ref{thm:distributionVC} provides a tool for proving upper bounds on the sample complexity of density estimation of any class $\fcal$.
%Namely, if $v$ is an upper bound for the VC-dimension of $\ycal(\fcal)$, then taking 
In view of Corollary~\ref{cor:vcsample}, to prove 
the sample complexity bounds for $\gcal_{d,1}$ and $\acal_{d,1}$, it remains to show upper bounds on the VC-dimensions of the Yatracos classes $\ycal(\gcal_{d,1})$
and $\ycal(\acal_{d,1})$.
We provide the proof for general Gaussians only; the proof for axis-aligned Gaussians is very similar.

For classes $\mathcal F$ and $\mathcal G$ of functions from $Z$ to $\R$, let
\[\operatorname{NN}(\mathcal{G})\coloneqq \left\{\{x : f(x)\geq0\}\textnormal{ for some }f\in \mathcal{G} \right\} \subseteq 2^Z,\]
and
\[\Delta \mathcal{F} \coloneqq \{f_1-f_2 : f_1, f_2 \in \mathcal{F} \}\:, \]
and notice that 
\[\mathcal{Y}(\fcal) = \operatorname{NN}(\Delta \fcal).\]
We upper bound the VC-dimension of
$\operatorname{NN}(\Delta \gcal_{d,1})$ via the following well known result in statistical learning theory, which first appeared in this form in~\cite[Theorem~7.2]{dudley_vectorspace} (see \cite[Lemma~4.2]{devroye_book} for a historical discussion).
\begin{theorem}[VC-dimension of vector spaces]
\label{dudley} Let ${\fcal}$ be an $n$-dimensional vector space of real-valued functions.
	Then $\vc(\operatorname{NN}(\fcal)) = n$.
\end{theorem}

Now let $h$ be the indicator function
for an arbitrary element in $\operatorname{NN}(f_1-f_2)$,
where $f_1,f_2 \in \gcal_{d,1}$.
Then $h$ is a $\{0,1\}$-valued function and we have:
\begin{align*}
h(x) & = \mathbbm{1}\{\mathcal{N}(\mu_1, \Sigma_1)(x) > \mathcal{N}(\mu_2, \Sigma_2)(x)\}\\
&  = \mathbbm{1}\{ \alpha_1 \exp(\frac{-1}{2}(x-\mu_1)^T\Sigma_1^{-1}(x-\mu_1)  ) >   \alpha_2 \exp(\frac{-1}{2}(x-\mu_2)^T\Sigma_2^{-1}(x-\mu_2)  ) \}\\
& = \mathbbm{1}\{ (x-\mu_1)^T\Sigma_1^{-1}(x-\mu_1) -(x-\mu_2)^T\Sigma_2^{-1}(x-\mu_2) - \log \frac{\alpha_2}{\alpha_1} >0 \}\:.
\end{align*}
The inner expression is a quadratic form, and the linear dimension of all quadratic functions is $O(d^2)$. 
Hence, by Theorem~\ref{dudley}, we have 
$\vc(\ycal(\gcal_{d,1})) = O(d^2)$.
Combined with Corollary~\ref{cor:vcsample}, this gives a sample complexity upper bound of $O(d^2/\eps^2)$ for learning 
$\gcal_{d,1}$, which is the main result of this section.
An $O(d/\eps^2)$ upper bound for $\acal_{d,1}$ can be proved similarly.

%Furthermore, for axis-aligned Gaussians, $\Sigma_1$ and $\Sigma_2$ are diagonal, and therefore, the inner function lies in an $O(d)$-dimensional space of functions spanned by $\{1,x_1,\dots,x_d,x_1^2,\dots,x_d^2\}$. 
%the required upper bound $d^2$) on the VC-dimension of the Yatracos classes. 

The problem with extending these results to mixtures of Gaussians is that it is not easy to bound the VC-dimension of the Yatrocas class of the family of mixtures of Gaussians.
It is an intriguing open problem whether
$\vc(\ycal(\gcal_{d,k})) = \widetilde{O}(k \vc(\ycal(\gcal_{d,1})))$.
One can also ask a more ambitious question:
is it true that for any class $\fcal$ of distributions, 
$\vc(\ycal(k\textnormal{-mix}(\fcal))) =\widetilde{O}(k \vc(\ycal(\fcal)))$ ?
We believe the answer to this latter question is no, but this is yet to be disproved.

The answers to these questions are unknown, so new ideas are required for density estimation of mixtures, described next.
First, in Section~\ref{sec:onedimensional} we discuss the 1-dimensional case, and we discuss the high-dimensional case in Sections~\ref{sec:mix} and~\ref{sec:compression}.

\section{Sample complexity upper bounds via piecewise polynomials}
\label{sec:onedimensional}
In this section we give  an $\widetilde{O}(k/\eps^2)$ upper bound for learning the class $\gcal_{1,k}$.
This was proved in~\cite{onedimensional},
which also gives a polynomial time algorithm for density estimation of this class.
The main idea is to approximate a Gaussian with a piecewise polynomial function.
For positive integers $D,t$, let $\pcal_{t,D}$ denote the class of density functions
that are piecewise polynomials with at most $t$ pieces, where each piece is a polynomial of degree at most $D$.
First, we give a sample complexity upper bound for learning $\pcal_{t,D}$ using the ideas from Section~\ref{sec:vc}.

We need to bound $\vc(\ycal(\pcal_{t,D}))=\vc(\nn(\Delta \pcal_{t,D}))$.
Note that $\Delta \pcal_{t,D} = \pcal_{t,D}$.
And since any $p \in \pcal_{t,D}$ has at most $tD$ roots and is continuous, any element in 
$\nn(\pcal_{t,D})$
is a union of at most $t D$ intervals.
The VC-dimension of the class of unions of at most $k$ intervals can be easily seen to be $O(k)$.
This gives $\vc(\ycal(\pcal_{t,D})) = O( t D)$, hence by Corollary~\ref{cor:vcsample}, the sample complexity of learning  $\pcal_{t,D}$ is $O(tD/\eps^2)$.

For any $g\in \gcal_{1,1}$ there exists $p\in \pcal_{3,O(\log(1/\eps))}$ with
$\|g - p\|_1 \leq \eps$.
(This is obtained by taking the Taylor polynomial for the main body of the Gaussian, and taking the zero polynomial for the two tails, see~\cite{onedimensional} for the details.)
Also, $k$-mix$(\pcal_{t,D}) \subseteq \pcal_{kt, D}$.
This implies that, for any $g\in \gcal_{1,k}$ there exists $p\in \pcal_{3k,O(\log(1/\eps))}$ with
$\|g - p\|_1 \leq \eps$.

Let $g$ be the target distribution.
Now, consider the empirical Yatracos minimizer (see Definition~\ref{eym}) for the class 
$\pcal_{3k,O(\log(1/\eps))}$.
Given samples from $g$, the minimizer `imagines' the samples are coming from $p$, and after taking 
 $O(k \log (1/\eps)/\eps^2)$ samples, outputs an estimate $\widehat{p}$ such that $\|\widehat{p}-p\|_1 \leq \eps$.
 Then, the triangle inequality gives 
\[\|\widehat{p}-g\|_1 \leq 
\|\widehat{p}-p\|_1 + \|\widehat{p}-g\| \leq 2\eps,\]
as required.

There is an issue with the above argument; our proof for Theorem~\ref{thm:distributionVC} assumed the samples are from a distribution in the known class of distributions ($\pcal_{3k,O(\log(1/\eps))}$ in this case),
whereas in this case, they are not.
However, one can amend the argument (by applying two careful triangle inequalities) to show that, if the samples are coming from a distribution $g$ that is not necessarily in $\fcal$, then with high probability the empirical Yatrocas minimizer outputs a distribution $L^{\mathcal{F}}(S)$ satisfying:
\begin{equation}\label{agnostic}
\|g - L^{\mathcal{F}}(S)\|_{TV} \leq 
3 \inf_{p \in \fcal} \|p - g\|_{TV}+
c \sqrt{\frac{v + \log \frac{1}{\delta}}{m}},
\end{equation}
which will be $O(\eps)$ in our case, as required
(see~\cite{ashtiani2017sample} for the proof of (\ref{agnostic})).
Such a result is called \emph{agnostic learning}, since it does not assume the target belongs to the known class, but only assumes it can be approximated well by some element of the class.

Unfortunately, the idea of piecewise polynomial approximation cannot be extended to higher dimensions, because to approximate a high-dimensional Gaussians, one needs a piecewise polynomial with either the degree or the number of pieces being exponential in the dimension.
The ideas for extending the bounds to higher dimensions are quite different and are described next.

\section{A generic upper bound for mixtures}
\label{sec:mix}

%In a recent work~\cite{ashtiani2017sample}, we proposed a sample-efficient method for learning with respect to mixture classes,
%assuming i.i.d.\ data and using $L_1$ distance as the measure of closeness. In particular, assuming the existence of a method for learning a base class $F$
%with sample complexity $m_F(\varepsilon)$, we provided a method for learning with respect to the class of $k$-mixtures of $F$ with sample complexity $O(k \log k \cdot  m_F(\varepsilon)/\varepsilon^2)$. This result significantly improves the
%state-of-the-art sample-complexity upper bounds for learning a variety of mixture classes (including
%Gaussian mixtures and log-concave mixtures).

We consider a more general problem in this section.
Assume that we have a method to learn an arbitrary class $\mathcal{F}$. Does this mean that we can learn $k$-mix$(\mathcal{F})$? And if so, what is the sample complexity of this task?
We give an affirmative answer to the first question, and provide a bound for sample complexity of learning $k$-mix$(\mathcal{F})$.
As an application of this general result, we give an upper bound for the case of mixtures of Gaussians in high-dimensions.
This section is based on~\cite{ashtiani2017sample}.
%\begin{theorem}
%	\label{thm:main}
%	Assume that $\mathcal{F}$ can be learned 
%	with sample complexity $m_{\mathcal{F}}(\eps, \delta) = {\lambda(\mathcal F,\delta)}/{\eps^\alpha}$ for some $C>0$, $\alpha \geq 1$ and some function
%	$\lambda(\mathcal F,\delta) = \Omega(\log(1/\delta))$. 
%	Then there exists a $3C$-agnostic PAC-learner for the class
%	$\mathcal{F}^k$  requiring $m_{\mathcal{F}^k}^{3C}(\eps, \delta) =$
%	
%	\[ O\left(\frac{\lambda (\mathcal F, \frac{\delta}{3k}) k\log k}{\eps^{\alpha+2}}\right)
%	=
%	O\left(\frac{k\log k \cdot m_{\mathcal F}(\eps, \frac{\delta}{3k}) }{\eps^{2}}\right)
%	\] 
%	samples.
%\end{theorem}
%
%Since a realizable PAC-learner is an $\infty$-agnostic PAC-learner, we immediately obtain the following corollary.

\begin{theorem}[sample complexity of learning mixtures]\label{thm:mixtures}
	Assume that $\mathcal{F}$ can be learned
	with sample complexity $m_{\mathcal{F}}(\eps, \delta) = {\lambda(\mathcal F,\delta)}/{\eps^\alpha}$ for some $\alpha \geq 1$ and some function
	$\lambda(\mathcal F,\delta) = \Omega(\log(1/\delta))$. 
	Then there exists a density estimation method for
	$k$-mix($\mathcal{F}$)  requiring
	\(\displaystyle 
	O\left({k\log k \cdot m_{\mathcal F}(\eps, \frac{\delta}{3k}) }\Big/{\eps^{2}}\right)
	\) samples.
\end{theorem}

One may wonder about tightness of this theorem.
	In Theorem~2 in \cite{spherical}, it is shown that 
	if $\mathcal F$ is the class of spherical Gaussians,  we have
	$m_{k\textnormal{-mix}(\mathcal{F})}(\eps, \delta) = \Omega(k m_{\mathcal F}(\eps, \delta) )$, therefore, the factor of $k$ is necessary in general. However, it is not clear whether the additional factor of
	$\log k /\eps^2$ in the theorem is tight.

If we apply this theorem to the class $\fcal = \gcal_{d,1}$, which has sample complexity $O(d^2/\eps^2)$ as proved in Section~\ref{sec:vc}, we immediately obtain an upper bound of $\widetilde{O}(kd^2/\eps^4)$ for the sample complexity of learning $\gcal_{d,k}$,
and a sample complexity upper bound of $\widetilde{O}(kd/\eps^4)$ for $\acal_{d,k}$.

We now give a sketch of the proof of Theorem~\ref{thm:mixtures}.
Suppose the target distribution is $g = \sum_{i=1}^{k} w_i g_i$, where each $g_i \in \fcal$.
The $w_i$ are  called the \emph{mixing weights}, and the $g_i$ are called the \emph{components}.
Consider a die with $k$ faces, such that when you roll it, the $i$th face has probability $w_i$ of coming.
To generate a point according to $g$, one can roll this die, and if the $i$th face comes, generate a point according to distribution $g_i$.
So, any i.i.d.\ sample generated from $g$ can be coloured with $k$ colours, such that almost a $w_i$ fraction of points have colour $i$, and the points with colour $i$ are i.i.d.\ distributed as $g_i$.

Now, if the colouring was given to the algorithm, there was a clear way to proceed: estimate each of the $g_i$ using the $\fcal$-learner, 
and estimate $w_i$ by the proportion of points with colour $i$, and then output the resulting mixture.
The issue is that the colouring is not given to the algorithm. But, in principle, it can do an exhaustive search over all possible colourings, and `choose the best one.'

More precisely, the algorithm has two main steps.
In the first step we generate a finite set of `candidate distributions,' such that at least one of them is $\eps$-close to $g$ in $L_1$ distance.
These candidates are of the form $\sum_{i=1}^{k} \widehat{w}_i \widehat{G}_i$, where the $\widehat{G}_i$'s are extracted from samples and are estimates for the real components $G_i$, 
and the $\widehat{w}_i$'s come from a fixed discretization of $\Delta_k$, and are estimates for the real mixing weights $w_i$.
In the second step, we take lots of  additional samples and use the following result
to choose the best one among them, giving a distribution that is $O(\eps)$-close to $g$.

The following theorem provides an algorithm that chooses the almost-best one among a finite set of candidate distributions.
It follows from \cite[Theorem~6.3]{devroye_book} and a standard Chernoff bound.
\begin{theorem}[handpicking from a finite set of candidates]
\label{thm:candidates}
	Suppose we are given $M$ candidate distributions $f_1,\dots,f_M$ and we have access to i.i.d.\ samples from an unknown distribution $g$.
	Then there exists an algorithm that given the $f_i$'s and  $\eps>0$, takes 
	$\log (3M^2/\delta)/2\eps^2$ samples from $g$, and with probability  $\geq 1-\delta$ outputs an index $j\in[M]$ such that 
	\[
	\|f_j-g\|_1 \leq 3 \min_{i\in[M]} \|f_i-g\|_1 + 4\eps \:.
	\]
\end{theorem}

We now analyze the sample complexity of our proposed algorithm.
First consider the simpler case that all mixing weights are equal to $1/k$.
To estimate $g$ within distance $\eps$, it suffices to estimate each $g_i$ within distance $\eps$.
Therefore, we need $ C k m_{\fcal}(\eps,\delta/k)$ total data points from $g$, with some large constant $C$, so that we get $ \Omega(m_{\fcal}(\eps,\delta/k))$
samples from each $g_i$ with probability $\geq 1-\delta$.
For each fixed way of colouring these 
$ C k m_{\fcal}(\eps,\delta/k)$ data points with $k$ colours, we provide the points of each colour to the $\fcal$-learner, and get an estimate $\widehat{g_i}$, and then we add $\sum_{i=1}^{k} \frac {1}{k} \widehat{g_i}$ to the set of candidate distributions (recall that we have assumed the mixture weights are $1/k$).
Hence, the total number of candidate distributions is 
$M= k^ {C k m_{\fcal}(\eps,\delta/k)}
= \exp (C k m_{\fcal}(\eps,\delta/k) \log k)$.

We now show that at least one of the candidate distributions is $O(\eps)$-close to the target.
Consider the colouring that assigns points to components correctly. Then the $\fcal$-learner would provide us with $(\widehat{g_i})_{i=1}^k$ that each is $\eps$-close to the corresponding $g_i$ with probability $\geq 1-\delta/k$.
So, by the union bound, they are simultaneously close, with probability $\geq 1-\delta$.
Thus, when we apply the algorithm of Theorem~\ref{thm:candidates},
with probability $\geq 1-\delta$ it provides us with one of the candidate distributions that is $7\eps$-close to the target.
The total sample complexity of the whole algorithm is thus
\[
O(k m_{\fcal}(\eps,\delta/k)) + O (\log M / \eps^2)
=O\left({k\log k \cdot m_{\mathcal F}(\eps, \frac{\delta}{3k}) }\Big/{\eps^{2}}\right),
\]
as required.

The general case of arbitrary mixing weights brings two challenges: first, we do not know the weights, and so we also do an exhaustive search over a finite fine grid on the simplex $\Delta_k$ to make sure that at least one of the candidate distributions also gets the weights right; it turns out that this does not increase the sample complexity by more than a constant factor.
The more important problem is that, for components with very small weight, we may not get enough samples if we take a total of $ C k m_{\fcal}(\eps,\delta/k)$ samples from the mixture.
The solution is to have different precision for different components: from small-weight components we will have fewer data points, so we will estimate them with larger error; this is compensated by the fact that their weight is small, so the effect of this error in the total estimation error can be controlled.
Here is the place that, for the error controlling calculations to work out, we need the technical conditions in the theorem, namely that
$m_{\mathcal{F}}(\eps, \delta) = {\lambda(\mathcal F,\delta)}/{\eps^\alpha}$ for some $\alpha \geq 1$ and that $\lambda(\mathcal F,\delta) = \Omega(\log(1/\delta))$.
See~\cite{ashtiani2017sample} for the details..
%A slightly simpler version of the algorithm below was provided in~\cite[Theorem~A.1]{gaussian_mixture}, which proved a sample complexity upper bound of $\widetilde{O}(k^3 d^2  / \eps^4)$ for this class;
%but it turns out that by a more careful analysis, the sample complexity bound of $\widetilde{O}(k d^2 / \eps^4)$ can be shown, as explained below.

\section{Sample complexity upper bounds via compression schemes}
\label{sec:compression}
The method of previous section would give sample complexity upper bounds of $\widetilde{O}(kd/\eps^4)$ for learning $\acal_{d,k}$,
and $\widetilde{O}(kd^2/\eps^4)$ for learning $\gcal_{d,k}$. In this section, we show how the work of ~\cite{ashtiani2017agnostic} improves these to 
$\widetilde{O}(kd/\eps^2)$ and
$\widetilde{O}(kd^2/\eps^2)$
using a technique called `compression.'
As before, let $\mathcal{F}$ be a class of distributions over a domain $Z$.

%Another technique for proving upper bounds for density estimation is via compression schemes, defined next.This technique gives tight (up to logarithmic factors) sample complexity ($kd/\eps^2$) for class $\acal_{d,k}$.

\begin{definition}[distribution decoder]
	A \emph{distribution decoder} for $\mathcal{F}$ is a deterministic function 
	$\mathcal{J}:\bigcup_{n=0}^{\infty} Z^n \times \bigcup_{n=0}^{\infty} \{0,1\}^n 
	\rightarrow \mathcal{F}$, which takes a finite sequence of elements of $Z$ and a finite sequence of bits, and outputs a member of $\mathcal{F}$. 
\end{definition}

%\begin{definition} \label{def_robustcompression}
%	[distribution compression schemes]
%	Let $\tau,t,m:(0,1)\rightarrow \mathbb{Z}_{\geq0}$
%	be functions, and let $r\geq0$.
%	We say that $\mathcal{F}$ admits $(\tau,t,m)$ $r$-robust compression if there exists a decoder $\mathcal{J}$ for $\mathcal{F}$ such that for any distribution $g \in \mathcal{F}$, and for any distribution $q$ on $Z$ 
%	with $\|g-q\|_1\leq r$, the following holds:
%	
%	\setlength{\leftskip}{0.5cm}
%	\setlength{\rightskip}{0.5cm}
%	For any $\eps \in (0,1)$, if $S\sim q^{m(\eps)}$, then with probability at least $2/3$, there exists a sequence $L$ of at most $d(\eps)$ elements of $S$, and a sequence $B$ of at most $t(\eps)$ bits, such that $\|\mathcal{J}(L,B)-g\|_1\leq \eps$.
%\end{definition}

\begin{definition} 	[distribution compression scheme]\label{def_compression}
	Let $\tau,t,m:(0,1)\rightarrow \mathbb{Z}_{\geq0}$
	be functions.
	We say  $\mathcal{F}$ admits \emph{$(\tau,t,m)$-compression} if there exists a decoder $\mathcal{J}$ for $\mathcal{F}$ such that for any distribution $g \in \mathcal{F}$ the following holds:
	
	\setlength{\leftskip}{0.5cm}
	\setlength{\rightskip}{0.5cm}
	For any $\eps \in (0,1)$, if $S\sim g^{m(\eps)}$, then with probability at least $2/3$, there exists a sequence $L$ of at most $\tau(\eps)$ elements of $S$, and a sequence $B$ of at most $t(\eps)$ bits, such that $\|\mathcal{J}(L,B)-g\|_1\leq \eps$.
\end{definition}

Essentially, the definition asserts that with high probability, there should be a (small) subset of $S$ and some (small number of) additional bits, from which $g$ can be reconstructed, or `decoded.'
We say that the distribution $g$ is `encoded'
with $L$ and $B$,
and in general we would like to have a compression scheme of a small size, for a reason that will be clarified soon.
%This compression scheme is called ``robust'' since one wants to reconstruct $g$ based on a sample that is generated from $q$ rather than $g$ itself. 

%\begin{remark}
%The core quantity in the analysis is $\tau+t$, i.e., the total number of bits and data points used for compression. Therefore, we sometimes use the notation of $(\tau',m)$-compression rather than the triplet notation, where $\tau' (\eps) = \tau (\eps)  + t (\eps) $.
%	%An ``efficient'' encoding will be one in which the size of the compression scheme, $(d+t)(\eps)$ is either bounded by a constant, or at most logarithmically dependent on $1/\eps$.
%\end{remark}

\begin{remark}\label{remark_probability}
	In the above definition  we required the probability of existence of $L$ and $B$ to be at least 2/3, but one can boost this probability to  $1-\delta$ by generating a sample of size $m(\eps)\log (1/ \delta)$.
\end{remark}

We next establish a connection between compression and learning, and also show some properties of compression schemes.
The proofs can be found in~\cite{ashtiani2017agnostic}.
\begin{lemma}[compression implies learning]
\label{thm:compression}
	Suppose $\mathcal{F}$ admits $(\tau,t,m)$-compression. 
	Let $\tau'(\eps)\coloneqq \tau(\eps/6)+t(\eps/6)$.
	Then $\mathcal{F}$ can be learned using 
	\begin{align*}
	O\left(
	m\left(\frac \eps 6\right) \log \frac{1}{\delta} + \frac{\tau'(\eps) \log (m\left( \frac \eps 6\right) \log(1/\delta)) + \log(1/\delta)}{\eps^2} 
	\right) = 
	\widetilde{O}
	\left(
	m\left(\frac \eps 6\right)  + \frac{\tau'(\eps)}{\eps^2} 
	\right) 
	\end{align*}
	samples.
\end{lemma}

The proof resembles  that for Theorem~\ref{thm:mixtures}: perform an exhaustive search over all possibilities for the `defining sequences' $L,B$ to generate some candidates; one of these candidates would give the `correct' $L,B$;
then apply Theorem~\ref{thm:candidates} to find the best one among the candidates.

Compression schemes have two nice closure properties.
First, if a class  $\fcal$ of distributions can be compressed,
then the class of distributions that are formed by taking products of distributions in $\fcal$ can also be compressed.
For a class $\fcal$ of distributions, we define
\( \fcal^d \coloneqq 
\left\{ \prod_{i=1}^{d} p_i : p_1,\dots,p_d \in \fcal \right\}.
\)
The proof of the following lemma is not too difficult.

\begin{lemma} [compressing product distributions]
	\label{lem:product_compress}
	If $\mathcal{F}$ admits $(\tau(\eps),t(\eps), m(\eps))$-compression, then $\mathcal{F}^d$ admits $(d \tau(\eps/d),d t(\eps/d), m(\eps/d)\log (3d))$-compression.
\end{lemma}

Second, if a class  $\fcal$ of distributions can be compressed,
then the class of mixtures of distributions in $\fcal$ can also be compressed.

% \begin{lemma} [compressing mixtures]
% 	\label{lem:compressmixtures}
% 	Let $m(\eps)$ be an invertible function and suppose
% 	that $x m^{-1}(k m (\eps/3) x)$
% 	is a concave function of $x$.
% 	If $\mathcal{F}$ admits 
% 	$(\tau,t,m)$-compression
% 	for $\tau$ and $t$ that are independent of $\eps$,
% 	then $k\textnormal{-mix}(\mathcal{F})$ admits 
% 	\((k \tau, kt+k\log_2 (4k/\eps), s(\eps))\)-compression, where
% 	\[
% 	s(\eps) = \max \left \{
% 	2 \log(6k) k m(\eps/3),
% 	48k \log (6k)/\eps
% 	\right\}.
% 	\]
% \end{lemma}

\begin{lemma} [compressing mixtures]
\label{lem:compressmixtures}
If $\fcal$ admits $(\tau(\eps), t(\eps), m(\eps))$-compression, then $\kmix(\fcal)$ admits $(k\tau(\eps/3), kt(\eps/3) + k \log_2 (4k/\eps)), 48 m(\eps/3) k \log(6k) / \eps)$-compression.
\end{lemma}

The proof of this lemma also resembles that for Theorem~\ref{thm:mixtures}: just take enough samples so that you have enough samples from each component, and also encode the weights of the mixture using the additional bits.

% \begin{remark}
% 	Any function of the form $m(\eps) = C\eps^{-\alpha}$ with $\alpha\geq1$ satisfies  the first assumption of the lemma.
% 	Most sample complexity functions are of this form.
% \end{remark}

One can easily show that a single 1-dimensional Gaussian can be compressed.

\begin{lemma}[compressing 1-dimensional Gaussians]
	\label{thm:1gaussiancompression}
	The class $\gcal_{1,1}$ admits $(2, 0, O(1/\eps))$-compression.
\end{lemma}

The proof is simple: given $C/\eps$ i.i.d.\ samples from a Gaussian $\ncal(\mu,\sigma^2)$ for a large constant $C$, with high probability there exists two generated points $x_1,x_2$ with 
$x_1 \in [\mu - (1+\eps)\sigma, \mu - (1-\eps)\sigma]$
and
$x_2 \in [\mu + (1-\eps)\sigma, \mu + (1+\eps)\sigma]$.
The decoder estimates $\widehat{\mu} = (x_1+x_2)/2$ and $\widehat{\sigma} = (x_2-x_1)/2$.
It is not hard to verify that $\|\ncal(\widehat{\mu},\widehat{\sigma})-\ncal(\mu,\sigma)\|_1=O(\eps)$.

Using the above properties, one can show a tight upper bound on the sample complexity of learning mixtures of axis-aligned Gaussians.

\begin{theorem}[learning mixtures of axis-aligned  Gaussians]
\label{maincor}
	The class $\acal_{d,k}$  of mixtures of $k$ axis-aligned Gaussians in $\mathbb{R}^d$ can be learned using 
	$\widetilde{O}(kd/\eps^2)$ many samples.
\end{theorem}

\begin{proof}
	By Lemma~\ref{thm:1gaussiancompression},
	$\gcal_{1,1}$ admits $(2,0,O(1/\eps))$-compression.
	By Lemma~\ref{lem:product_compress}, the class $\acal_{d,1}=\gcal_{1,1}^d$ admits $(O(d),0,O(d \log (3d)/\eps))$-compression.
	Then by Lemma~\ref{lem:compressmixtures},
	the class $k$-mix($\acal_{d,1}$) $=\acal_{d,k}$ admits
	$(O(kd), O(k\log(k/\eps)),O(kd \log(6k) \log(3d) / \eps^2)$-compression.
Applying Lemma~\ref{thm:compression} gives the theorem.
\end{proof}

Using more complicated arguments, one can also
prove that the class of $d$-dimensional Gaussian distributions admits $\big(O(d\log (2d)), \allowbreak O(d^2 \log (2d) \log(d/\eps)), \allowbreak O(d \log (2d)) \big)$ compression.
The high level idea is that by generating $O(d \log (2d)) $ samples from a Gaussian, one gets a rough sketch of the geometry of the Gaussian.
In particular, the convex hull of the points drawn from a Gaussian enclose an ellipsoid 
centred at the mean and whose principal axes are the eigenvectors of the covariance matrix.
Using ideas from 
convex geometry
and random matrix theory, one can in fact encode the centre of the ellipsoid \emph{and} the principal axes using a convex combination of these samples.
Then we discretize the coefficients and obtain an approximate encoding.
See~\cite{ashtiani2017agnostic} for details.

\begin{lemma}[compressing high-dimensional Gaussians]
\label{lem:coreformixtures}
For any positive integer $d$, the class of $d$-dimensional Gaussians admits an $\big(O(d\log (2d)),\allowbreak O(d^2 \log (2d) \log(d/\eps)),\allowbreak O(d \log (2d)) \big)$-compression scheme.
\end{lemma}

In the next section we will see the following bound is tight up to polylogarithmic factors.

\begin{theorem}[learning mixtures of Gaussians]
The class of $k$-mixtures of $d$-dimensional Gaussians can be learned using $\widetilde{O}(kd^2/\eps^2)$ samples.
\end{theorem}

\begin{proof} 
Combining Lemma~\ref{lem:coreformixtures} and Lemma~\ref{lem:compressmixtures}
implies that the class of $k$-mixtures of $d$-dimensional Gaussians admits a
$$\big(\, O(kd\log (2d)) ,\,
O(kd^2 \log (2d)\log(d/\eps)
+ k \log (k/\eps)) ,\,
O(dk \log k \log(2d)/\eps) \,\big)
$$ compression scheme.
Applying Lemma~\ref{thm:compression} with
$m(\eps)=\widetilde{O}(d k/\eps)$ and $\tau'(\eps)=\widetilde{O}(kd^2 )$
gives  the sample complexity of learning this class is $\widetilde{O}(kd^2/\eps^2)$, completing the proof of
the theorem.
\end{proof}

\section{Lower bounds via Fano's inequality}
\label{sec:lower}

In the previous sections we gave several techniques for upper bounding the sample complexity of density estimation.
This survey would feel incomplete if we do not discuss at least one technique for proving lower bounds.
Note that each of our upper bounds holds uniformly over a class: the sample complexity does not depend on the specific distribution.
Similarly, the lower bounds we discuss in this section also hold for a class of distribution rather than for a specific distribution.
Such a bound is called a minimax lower bound in the statistics literature, and a worst-case lower bound in the computer science literature.

In this section, which is based on~\cite{spherical,ashtiani2017agnostic}, we give a sample complexity lower bound of $\widetilde\Omega(kd/\eps^2)$ for $\acal_{d,k}$,
and a lower bound of $\widetilde\Omega(kd^2/\eps^2)$ for $\gcal_{d,k}$.
That is, we show that any density estimation method that learns the class $\gcal_{d,k}$ 
uniformly, in the sense of Definition~\ref{def:density_estimation},
must have a sample complexity of $\widetilde\Omega(kd^2/\eps^2)$.
This shows the density estimation method described in the previous section has optimal sample complexity, up to polylogarithmic factors.

We will need the definition of the Kullback-Leibler divergence (KL-divergence, also called the relative entropy) between two distributions.

\begin{definition}[Kullback-Leibler divergence]
Let $f$ and $g$ be densities over domain $Z$.
Their \emph{KL-divergence} is defined as
\[\kl(f \parallel g)=
\int_Z f (x) \log \frac{f(x)}{g(x)} \textnormal{ d}x.
\]
\end{definition}

The KL-divergence is a measure of closeness between distributions.
It is always non-negative, and is zero if and only if the two distributions are equal almost everywhere.
However, it is not a metric, since it is not symmetric, and it can be $+\infty$.

The proof of the following lemma, which is called the `generalized Fano's inequality,' uses Fano's inequality from information theory.
It was first proved in~\cite[page 77]{devroye_density_estimation_first}.
We write here a slightly stronger version, which appears in~\cite[Lemma 3]{bin_yu}.

\begin{lemma}[generalized Fano's inequality]
	Suppose we have $M>1$ distributions $f_1,\dots,f_M$ with
	\[
	\kl(f_i \parallel f_j) \leq \beta \textnormal{ and }
	\|f_i - f_j\|_1 \geq \alpha
	\qquad
	\forall i\neq j\in[M].\]
	Consider any density estimation method that gets $n$ i.i.d.\ samples from some $f_i$, and outputs an estimate $\widehat{f}$ (the method does not know $i$). 
	For each $i$, define $e_i$ as follows:
	assume the method receives samples from $f_i$, and outputs $\widehat{f}$. Then $e_i\coloneqq \mathbf{E}\|f_i-\widehat{f}\|_1$.
	Then, we have
	\[ \max_i e_i \geq \alpha (\log M - n\beta + \log 2) / (2 \log M) \:.\]
\end{lemma}

This immediately leads to the following sample complexity lower bound for learning a class $\fcal$.

\begin{corollary}\label{cor:fano}
	Suppose for all small enough $\eps>0$ there exist $M$ densities $f_1,\dots,f_M \in \fcal$ with
	\[
	\kl(f_i \parallel f_j) =O(\kappa(\eps)) \textnormal{ and }
	\|f_i - f_j\|_1 =\Omega(\eps)
	\qquad
	\forall i\neq j\in[M].\]	
	Then the sample complexity of learning $\fcal$  is $\Omega((\log M) / \kappa(\eps) \log(1/\eps) )$.
\end{corollary}

We start with describing the lower bound construction
which gives a sample complexity lower bound of $\widetilde\Omega(kd/\eps^2)$ for $\acal_{d,k}$.
This lower bound was proved in~\cite{spherical}.

We claim it suffices to give a lower bound of $\Omega(d/\eps^2)$ for $\acal_{d,1}$.
For, consider a mixture of axis-aligned Gaussians whose components are extremely far away, such that the total variation distance between any two components is very close to 1.
To learn the mixture distribution, one needs to learn each component.
But each data point can help in learning \emph{one} of the $k$ components.
Since for learning any of the components one needs $\Omega(d/\eps^2)$ samples,
one will need 
$\Omega(kd/\eps^2)$ samples to learn the mixture.
Some nontrivial work has to be done to make this intuitive argument rigorous, but we omit that, and focus on proving the lower bound of $\Omega(d/\eps^2)$ for $\acal_{d,1}$.

%The proof uses Fano's inequality from information theory.

Let $M \coloneqq 2^{d/5}$.
To prove a lower bound of $\Omega(d/\eps^2)$ for $\acal_{d,1}$, we will build $M$ densities $f_1,\dots,f_{M} \in \acal_{d,1}$ satisfying the conditions of the corollary.

By the Gilbert-Varshamov bound in coding theory, there exists $2^{d/5}$ elements in $\{0,1\}^d$ such that any two of them differ in at least $d/5$ components. 
(To see this, note that the size of a Hamming ball of radius $d/5$ is 
$$\sum_{j=0}^{d/5} \binom{d}{j} \leq \left(\frac{ed}{d/5}\right)^{d/5}
=(5e)^{d/5} < 2^{4d/5};$$
hence one can start from an empty set $S$,
then add elements from $\{0,1\}^d$ one by one to $S$, and delete the Hamming ball of each added element.
So long as $S$ has less than $2^{d/5}$ elements,
the number of deleted elements is less than $2^{d/5}\times 2^{4d/5} = 2^d$, so there are still undeleted elements in $\{0,1\}^d$, and one still can add more elements to $S$.)
Call them $\mu_1,\dots,\mu_M \in \mathbb{R}^d$,
and let $f_i \coloneqq \ncal(\mu_i \eps / \sqrt d, I_d)$.
The densities $f_1,\dots,f_M$ are Gaussians with identity covariance matrix,
with their means are chosen carefully at vertices of $d$-dimensional hypercube with side length $\eps/\sqrt d$.

The KL-divergence between two general Gaussians
$\ncal(\mu,\Sigma)$ and $\ncal(\mu',\Sigma')$ is given by (see, e.g., \cite[Section~9]{duchi})
\[\kl(\ncal(\mu,\Sigma)\parallel\ncal(\mu',\Sigma')) 
	={
		\frac{1}{2} \left( \tr(\Sigma^{-1}\Sigma' - I)
		+ (\mu-\mu')^T \Sigma^{-1} (\mu-\mu')^T
		- \ln \left( \frac {\det (\Sigma')}{\det (\Sigma)}\right)
		\right)},\]
thus, for $i,j\in [M]$ we get
\begin{align*}
\kl(f_i \parallel f_j)  &= \kl(\ncal(\mu_i \eps / \sqrt d,I_d),\ncal(\mu_j \eps / \sqrt d,I_d))  \\
	&={
		\frac{1}{2} \left( \|\mu_i-\mu_j\|_2^2 \eps^2 / d
		\right)}
		\leq 
		\frac{1}{2} \left( d \eps^2 / d	\right) = \eps^2/2,
\end{align*}
		
as required.

Fix distinct $i$ and $j$.
We now lower bound the $L_1$ distance 
between $f_i$ and $f_j$.
Let $X_i \sim f_i$ and $X_j \sim f_j$.
Any two $\mu_i$ and $\mu_j$ differ in at least $d/5$ coordinates.
Fix such $d/5$ coordinates,
and, without loss of generality, assume that in $\mu_i$ these coordinates are 0, and they are 1 in $\mu_j$.
If we project $X_i$ onto one such coordinate, we get an $N(0,1)$ random variable,
so the sum over these coordinates of $X_i$ has distribution $\ncal(0, d/5)$.
Similarly, if we project $X_j$ onto one such coordinate, we get an $N(\eps / \sqrt d,1)$ random variable,
so the sum over these coordinates of $X_j$ has distribution $\ncal(\eps  \sqrt d, d/8)$.
The total variation distance between
$\ncal(0, d/8)$ and $\ncal(\eps  \sqrt d, d/8)$ equals
the total variation distance between
$\ncal(0, 1)$ and $\ncal(\sqrt 8 \eps, 1)$,
which is $\Omega(\eps)$ (see Lemma~\ref{lem:tvnormals} below).
Hence, the total variation distance between $f_i$ and $f_j$ is also $\Omega(\eps)$, as required. 

%\hassan{notation for random variable vs density throughout the paper}

\begin{lemma}\label{lem:tvnormals}
Let $\eps \in [0,1]$. Then,
\[ \| \ncal(0,1) - \ncal(\eps,1) \|_1 \geq \eps / 5 .\]
\end{lemma}
\begin{proof}
\begin{align*}
\| \ncal(0,1) - \ncal(\eps,1) \|_1 &
= \frac{2}{\sqrt{2\pi}} \int_{\eps/2}^{\infty} e^{-(x-\eps)^2/2} - e^{-x^2/2} dx\\
&= \frac{2}{\sqrt{2\pi}} \int_{\eps/2}^{\infty}
e^{-x^2/2} \left( e^{-\eps^2/2+x\eps}-1 \right) dx\\
&\geq
\frac{2}{\sqrt{2\pi}} \int_{\eps/2}^{\infty}
e^{-x^2/2} \left(-\eps^2/2+x\eps \right) dx\\
&=
\frac{-\eps^2}{\sqrt{2\pi}} \int_{\eps/2}^{\infty}
e^{-x^2/2}  dx
+
\frac{2\eps}{\sqrt{2\pi}} \int_{\eps/2}^{\infty}
e^{-x^2/2} x  dx\\
& =
-\eps^2 \mathbf{Pr}[N(0,1)\geq\eps/2]
+
\frac{2\eps}{\sqrt{2\pi}} e^{-(\eps/2)^2/2} \\
& = \eps \left( 
\frac{2 e^{-\eps^2/8} }{\sqrt{2\pi}} -  \eps \mathbf{Pr}[N(0,1)\geq\eps/2]
\right) \geq \eps/5,
\end{align*}
since $\mathbf{Pr}[N(0,1)\geq\eps/2] \leq 1/2$ and $\eps\in[0,1]$, completing the proof.
\end{proof}

We now describe the lower bound construction
which gives a sample complexity lower bound of $\widetilde\Omega(kd^2/\eps^2)$ for $\gcal_{d,k}$. This lower bound was proved in~\cite{ashtiani2017agnostic}.
Again it suffices to give a lower bound of $\Omega(d^2/\eps^2)$ for $\gcal_{d,1}$.
Let $r = 9$ and $\lambda =  \eps  \log(1/\eps)/\sqrt{d}$.
Guided by Corollary~\ref{cor:fano}, we will build $2^{\Omega(d^2)}$ Gaussian distributions of the form $f_a\coloneqq \ncal(0, \Sigma_a)$ where $\Sigma_a = I_d + \lambda U_a U_a\transpose$, where each $U_a$ is a $d\times d/r$ matrix with orthonormal columns.
To apply Corollary~\ref{cor:fano}, we need to give an upper bound on the KL-divergence between any two $f_a$ and $f_b$, and a lower bound on their total variation distance.
Upper bounding the KL divergence is  easy:
by the KL-divergence formula and since $\|U_a\transpose U_b\|_F^2\geq0$,
\begin{align*}
2\DKL{f_a}{f_b}
&=
 \tr(\Sigma_a^{-1}\Sigma_b - I)
=\tr( (I - \frac{\lambda}{1+\lambda} U_a U_a\transpose) (I + \lambda U_b U_b\transpose) - I)\\
& = 
\tr( \lambda U_b U_b\transpose - \frac{\lambda}{1+\lambda} U_a U_a\transpose  - \frac{\lambda^2}{1+\lambda}U_a U_a\transpose U_b U_b\transpose  ) \\
& =  \lambda (d/r) - \frac{\lambda}{1+\lambda} (d/r)  - 
\frac{\lambda^2}{1+\lambda} \|U_a\transpose U_b\|_F^2 \\
& \leq \frac{\lambda^2 d}{r+r\lambda} \leq \lambda^2 d / (2r)=O(\eps^2 \log^2(1/\eps)),
\end{align*}
as required.

Our next goal is to give a lower bound on the total variation distance between $f_a$ and $f_b$.
For this, we would like the matrices $\{U_a\}$ to be ``spread out,'' in the sense that their columns should be nearly orthogonal.
It is possible to show that if we choose the $U_a$ randomly, we can achieve $\|U_a\transpose U_b \|_F^2 \leq \frac{d}{2r}$ for any $a\neq b$.
Then, if $S_a$ is the subspace spanned by the columns of $U_a$, then we expect that a Gaussian drawn from $\ncal(0, \Sigma_a)$ should have a slightly larger projection onto $S_a$ then a Gaussian drawn from $\ncal(0, \Sigma_b)$.
This will then allow us to give a lower bound on the total variation distance between $\ncal(0, \Sigma_a)$ and $\ncal(0, \Sigma_b)$.
More precisely, one can show that $\|U_a\transpose U_b \|_F^2 \leq \frac{d}{2r}$ implies
 $\DTV{f_a}{ f_b} = \Omega\left( \frac{\lambda \sqrt{d}}{\log(1 / \lambda\sqrt{d})} \right) = \Omega(\eps)$, completing the proof.

We mentioned here just one lower bound technique based on Fano's inequality. 
There are at least two other methods for proving lower bounds, see~\cite{bin_yu}.

\section{Concluding remarks}
\label{sec:remarks}

\textbf{Characterizing the sample complexity of a class.}
An insight from supervised learning theory is that the sample complexity of learning a class (of concepts, functions, or distributions) must be proportional to 
the intrinsic dimension of the class divided by $\varepsilon^2$, where $\varepsilon$ is the error tolerance. (We would expect this dimension to be  equal the number of parameters needed to describe an object in that class using its `natural' parametrization, or the `degrees of freedom' of an object in the class.)
One challenge in learning theory is to formally define this dimension.
For the case of binary classification, the intrinsic dimension is captured by the VC-dimension of the concept class (see~\cite{vc,Blumer:1989}).

In the case of distribution learning/density estimation, a candidate for quantifying the dimension of a class, is the VC-dimension of its Yatrocas class.
However, it is not hard to come up with examples where this VC-dimension is infinite while the class can be learned with finite samples.
A second candidate is using covering numbers. This method also does not work;
for instance, the class of Gaussians do not have a finite covering number, yet its sample complexity is finite.

We showed in Section~\ref{sec:compression} that if a class of distribution has a compression scheme of size $D$, then its sample complexity is $D/\varepsilon^2$, up to logarithmic factors.
Hence, the size of the smallest compression is one candidate for capturing the intrinsic dimension of a distribution class.
As we have shown here, for high-dimensional Gaussians and their mixtures, the smallest compression size captures the sample complexity up to logarithmic factors
(and it is also equal to the number of parameters needed to describe a distribution in the class using the mean and the covariance matrix).
It is an intriguing open question whether distribution learning and  compressibility are equivalent (such a statement has recently been shown to be true in the setting of binary classification~\cite{Moran:2016}).

\paragraph{\textbf{Robust density estimation, or agnostic learning}.}
We have assumed that the target belongs to the prescribed class of distributions.
In practice, this is rarely the case.
Fortunately, it turns out that all the methods we discussed here can be extended to the case where the target is not necessarily in the class, but is \emph{close} (in total variation distance) to some member of this class, albeit with some loss in the approximation error.
An example of a guarantee that can be given in this setting is Equation~(\ref{agnostic}).
We refer the reader to the relevant papers for details.

\paragraph{\textbf{Discrete distributions and covering numbers.}} We have focused on continuous distributions in this survey.
Learning discrete distributions is a whole other world and many exciting techniques have been developed for those classes. One important technique is finding an $\eps$-net for the target class and then applying Theorem~\ref{thm:candidates}.
See~\cite{Diakonikolas2016} and the references therein.

\paragraph{\textbf{Kernel density estimation.}}
A  popular method for density estimation in practice is kernel density estimation (see, e.g., \cite[Chapter~9]{devroye_book}).
The few proven convergence rate results for this method, require certain smoothness assumptions on the class of densities (e.g., \cite[Theorem~9.5]{devroye_book}). The class of  Gaussians is not universally Lipschitz and do not satisfy these assumptions, so these results do not apply.
Indeed we believe that any kernel-based method for learning a high dimensional Gaussian must have sample complexity exponential in the dimension, but this is yet to be proved.

\paragraph{\textbf{Variants of the problem.}}
There are many natural and important variants of the problems presented above. 
What if instead of the $L_1$ distance, we consider the KL-divergence, or the $L_2$ distance, as the measure of closeness? The $L_1$ results do not carry over immediately, and several new ideas are needed.
Also, in practice the i.i.d.\ assumption is not realistic, and an interesting direction is to extend the above results to settings where some correlation among the input data points is possible.

\paragraph{\textbf{Computational complexity.}}
Designing efficient algorithms for distribution learning is crucial for practical applications; the sample complexity does not always capture the computational difficulty.
While we have totally ignored computational issues in this survey for brevity, a whole complexity theory can be developed around the task of distribution learning: which classes are `easy' to learn, and which ones are hard?
We refer the reader to~\cite{gaussian_mixture} and the references therein.

\paragraph{\textbf{Parameter estimation.}}
In \emph{parameter estimation},
which has been greatly popular in the theoretical computer science community recently (see, e.g.,~\cite{dasgupta1999learning,Belkin,moitravaliant}), the goal is to identify the `parameters' of the target distribution, for example, the mixing weights and the parameters of the Gaussians, up to a desired accuracy.
Parameter estimation is a more difficult problem than density estimation, and any algorithm for parameter estimation requires some separability assumptions for the target Gaussians, whereas for density estimation no such assumption is needed. E.g., consider the case that $k = 2$ and the two
components are identical; then there is no way to learn their mixing weights.


\begin{thebibliography}{10}

\bibitem{AB99}
Martin Anthony and Peter Bartlett.
\newblock {\em Neural network learning: theoretical foundations}.
\newblock Cambridge University Press, 1999.

\bibitem{ashtiani2017agnostic}
Hassan Ashtiani, Shai Ben-David, Nick Harvey, Christopher Liaw, Abbas
  Mehrabian, and Yaniv Plan.
\newblock Settling the sample complexity for learning mixtures of gaussians.
\newblock {\em arXiv preprint arXiv:1710.05209}, 2018.

\bibitem{ashtiani2017sample}
Hassan Ashtiani, Shai Ben-David, and Abbas Mehrabian.
\newblock Sample-efficient learning of mixtures.
\newblock {\em arXiv preprint arXiv:1706.01596}, 2017.
\newblock Accepted for presentation in the AAAI Conference on Artificial
  Intelligence (AAAI'18).

\bibitem{Belkin}
Mikhail Belkin and Kaushik Sinha.
\newblock Polynomial learning of distribution families.
\newblock In {\em Proceedings of the 2010 IEEE 51st Annual Symposium on
  Foundations of Computer Science}, FOCS '10, pages 103--112, Washington, DC,
  USA, 2010. IEEE Computer Society.

\bibitem{Blumer:1989}
Anselm Blumer, Andrzej Ehrenfeucht, David Haussler, and Manfred~K. Warmuth.
\newblock Learnability and the {V}apnik-{C}hervonenkis dimension.
\newblock {\em J. ACM}, 36(4):929--965, October 1989.

\bibitem{onedimensional}
Siu-On Chan, Ilias Diakonikolas, Rocco~A. Servedio, and Xiaorui Sun.
\newblock Efficient density estimation via piecewise polynomial approximation.
\newblock In {\em Proceedings of the Forty-sixth Annual ACM Symposium on Theory
  of Computing}, STOC '14, pages 604--613, New York, NY, USA, 2014. ACM.

\bibitem{dasgupta1999learning}
Sanjoy Dasgupta.
\newblock Learning mixtures of gaussians.
\newblock In {\em 40th Annual Symposium on Foundations of Computer Science,
  1999.}, pages 634--644. IEEE, 1999.

\bibitem{devroye_density_estimation_first}
Luc Devroye.
\newblock {\em A course in density estimation}, volume~14 of {\em Progress in
  Probability and Statistics}.
\newblock Birkh\"auser Boston, Inc., Boston, MA, 1987.

\bibitem{devroye_book}
Luc Devroye and G\'abor Lugosi.
\newblock {\em Combinatorial methods in density estimation}.
\newblock Springer Series in Statistics. Springer-Verlag, New York, 2001.

\bibitem{gaussian_mixture}
I.~Diakonikolas, D.~M. Kane, and A.~Stewart.
\newblock Statistical query lower bounds for robust estimation of
  high-dimensional {G}aussians and {G}aussian mixtures.
\newblock In {\em 2017 IEEE 58th Annual Symposium on Foundations of Computer
  Science (FOCS)}, pages 73--84, Oct 2017.
\newblock Available on arXiv:1611.03473 [cs.LG].

\bibitem{Diakonikolas2016}
Ilias Diakonikolas.
\newblock {Learning Structured Distributions}.
\newblock In {Peter B{\"{u}}hlmann}, {Petros Drineas}, {Michael Kane}, and
  {Mark van der Laan}, editors, {\em Handbook of Big Data}, chapter~15, pages
  267--283. Chapman and Hall/CRC, 2016.

\bibitem{logconcave}
Ilias Diakonikolas, Daniel~M Kane, and Alistair Stewart.
\newblock Learning multivariate log-concave distributions.
\newblock In {\em Proceedings of Machine Learning Research}, volume~65 of {\em
  COLT'17}, pages 1–--17, 2017.

\bibitem{duchi}
John Duchi.
\newblock Derivations for linear algebra and optimization.
\newblock Available at
  \url{http://stanford.edu/~jduchi/projects/general_notes.pdf}. Accessed on
  20/12/2017.

\bibitem{dudley_vectorspace}
R.~M. Dudley.
\newblock Central limit theorems for empirical measures.
\newblock {\em Ann. Probab.}, 6(6):899--929, 12 1978.

\bibitem{deep_learning_book}
Ian~J. Goodfellow, Yoshua Bengio, and Aaron~C. Courville.
\newblock {\em Deep Learning}.
\newblock MIT Press, Cambridge, MA, 2016.

\bibitem{jordan96neural}
Michael~I. Jordan and Christopher~M. Bishop.
\newblock Neural networks.
\newblock In Teofilo Gonzalez, Jorge Diaz-Herrera, and Allen Tucker, editors,
  {\em Computing Handbook, Third Edition: Computer Science and Software
  Engineering}. Chapman and Hall/CRC, Boca Raton, FL, 2014.
\newblock Article 42.

\bibitem{moitravaliant}
Ankur Moitra and Gregory Valiant.
\newblock Settling the polynomial learnability of mixtures of gaussians.
\newblock In {\em Proceedings of the 2010 IEEE 51st Annual Symposium on
  Foundations of Computer Science}, FOCS '10, pages 93--102, Washington, DC,
  USA, 2010. IEEE Computer Society.

\bibitem{Moran:2016}
Shay Moran and Amir Yehudayoff.
\newblock Sample compression schemes for vc classes.
\newblock {\em J. ACM}, 63(3):21:1--21:10, June 2016.

\bibitem{epsnet}
Nabil~H. Mustafa and Kasturi Varadarajan.
\newblock Epsilon-approximations \& epsilon-nets.
\newblock In Jacob~E. Goodman, Joseph O'Rourke, and Csaba~D. T{\'{o}}th,
  editors, {\em Handbook of Discrete and Computational Geometry}. CRC Press
  LLC, Boca Raton, FL, third edition, 2017.
\newblock Chapter 47.

\bibitem{Shalev-Shwartz2014}
Shai Shalev-Shwartz and Shai Ben-David.
\newblock {\em Understanding Machine Learning: From Theory to Algorithms}.
\newblock Cambridge University Press, New York, NY, USA, 2014.

\bibitem{spherical}
Ananda~Theertha Suresh, Alon Orlitsky, Jayadev Acharya, and Ashkan Jafarpour.
\newblock Near-optimal-sample estimators for spherical gaussian mixtures.
\newblock In Z.~Ghahramani, M.~Welling, C.~Cortes, N.~D. Lawrence, and K.~Q.
  Weinberger, editors, {\em Advances in Neural Information Processing Systems
  27}, pages 1395--1403. Curran Associates, Inc., 2014.

\bibitem{tal}
M.~Talagrand.
\newblock Sharper bounds for {G}aussian and empirical processes.
\newblock {\em Ann. Probab.}, 22(1):28--76, 1994.

\bibitem{vc}
V.~N. Vapnik and A.~Ya. Chervonenkis.
\newblock On the uniform convergence of relative frequencies of events to their
  probabilities.
\newblock {\em Theory of Probability \& Its Applications}, 16(2):264--280,
  1971.

\bibitem{bin_yu}
Bin Yu.
\newblock Assouad, {F}ano, and {L}e {C}am.
\newblock In {\em Festschrift for {L}ucien {L}e {C}am}, pages 423--435.
  Springer, New York, 1997.

\end{thebibliography}
\end{document}